\documentclass[12pt,reqno]{amsart}

\usepackage{mathrsfs}
\usepackage{amsfonts}
\usepackage{amssymb}
\usepackage{amsthm}
\usepackage{amsmath}
\usepackage{latexsym}
\usepackage{tikz}
\usepackage{cite}
\allowdisplaybreaks

\newtheorem{theorem}{Theorem}[section]

\newtheorem{lemma}[theorem]{Lemma}

\numberwithin{equation}{section}

\newcommand{\beq}{\begin{equation}}
\newcommand{\eeq}{\end{equation}}

\newcommand{\Rmnum}[1]{\expandafter\@slowromancap\romannumeral #1@}

\def\na{\nabla}

\newcommand{\p}{\partial}

\newcommand{\ben}{\begin{eqnarray}}
\newcommand{\een}{\end{eqnarray}}
\newcommand{\beno}{\begin{eqnarray*}}
\newcommand{\eeno}{\end{eqnarray*}}

\begin{document}

\title[Global  solutions of 3D  MHD equations with mixed dissipation]
{Global solutions of 3D incompressible MHD system with mixed partial dissipation
and magnetic diffusion near an equilibrium}

\author[Jiahong Wu and Yi Zhu ]{Jiahong Wu $^{1}$ and Yi Zhu $^{2}$}

\address{$^1$ Department of Mathematics, Oklahoma State University, Stillwater, OK 74078, United States }

\email{jiahong.wu@okstate.edu}

\address{$^2$ Department of Mathematics, East China University of Science and Technology, Shanghai 200237,  P.R. China}

\email{zhuyim@ecust.edu.cn}

\date{}

\begin{abstract}
This paper focuses on the 3D incompressible magnetohydrodynamic (MHD) equations with
mixed partial dissipation and magnetic diffusion. Our main result assesses the
global stability of perturbations near the steady solution given by a background magnetic field. The stability problem on the MHD equations with partial or no dissipation has attracted considerable interests recently and there are substantial developments. The new stability result presented here is among the very few stability conclusions currently available for ideal or partially dissipated MHD equations. As
 a special consequence of the techniques introduced in this paper, we obtain the small data global well-posedness for the 3D incompressible Navier-Stokes equations without vertical dissipation.
\end{abstract}

\maketitle

\section{introduction}

The magnetohydrodynamic (MHD) equations reflect the basic physics laws governing the
motion of electrically conducting fluids such as plasmas, liquid metals, and electrolytes. The velocity field obeys the Navier-Stokes equations with
Lorentz forcing
generated by the magnetic field while the magnetic field satisfies the Maxwell's equations of electromagnetism. The MHD equations have played pivotal roles in the study of many phenomena in geophysics,
astrophysics, cosmology and engineering (see, e.g., 
\cite{Bis,Davi,Pri}).

\vskip .1in
The MHD equations are also mathematically significant. The MHD equations share similarities with the Navier-Stokes equations, but they contain much richer structures than the Navier-Stokes
equations. They are not merely a combination of two parallel
Navier-Stokes type equations but an interactive and integrated system. Their distinctive features make analytic studies a great challenge but offer new opportunities.

\vskip .1in
Two fundamental problems on the MHD equations have recently attracted considerable interests. The first is the existence and uniqueness of solutions while the
second concerns the stability of perturbations near physically relevant equilibrium.
There have been substantial developments on these problems, especially on  those MHD systems with only partial or fractional dissipation.

\vskip .1in
This paper focuses on a stability problem concerning the following 3D incompressible MHD system with mixed partial dissipation and magnetic diffusion,
\begin{equation}\label{mhd0}
\begin{cases}
\p_t u + u \cdot \nabla u - \partial_1^2 u -\partial_2^2 u + \nabla P = B \cdot \nabla B,\\
\p_t B + u \cdot \nabla B - \partial_3^2 B = B \cdot \nabla u, \\
\nabla \cdot u = \nabla \cdot B = 0,
\end{cases}
\end{equation}
where $u$ represents the velocity field, $P$ the total pressure and $B$ the magnetic field. (\ref{mhd0}) may be physically relevant when the vertical dissipation and horizontal magnetic diffusion can be ignored.
It is clear that a special solution of (\ref{mhd0}) is given by the zero velocity field and the background
magnetic field $B^{(0)} =e_1$, where $e_1=(1,0, 0)$. The perturbation $(u, b)$  around this equilibrium with $b= B - e_1$ obeys
\begin{equation}\label{mhd}
\begin{cases}
\p_t u + u \cdot \nabla u - \Delta_h u + \nabla P = b \cdot \nabla b + \partial_1 b,\\
\p_t b + u \cdot \nabla b - \partial_3^2 b = b \cdot \nabla u + \partial_1 u, \\
\nabla \cdot u = \nabla \cdot b = 0.
\end{cases}
\end{equation}
where, for notational convenience, we have written
$$
\Delta_h = \partial_1^2 + \partial_2^2
$$
and we shall also write $\na_h =(\p_1, \p_2)$.

\vskip .1in
This paper aims at the stability problem on the perturbation $(u, b)$. Equivalently, we establish a small data global well-posedness result for (\ref{mhd}) supplemented with the initial condition
$$
u(x, 0)  = u_0(x), \qquad b(x,0) =b_0(x).
$$
Our main result can be stated as follows. The notation $A \lesssim D$ means $A \le C\, D$ for a pure constant $C$.
\begin{theorem} \label{good}
	Consider (\ref{mhd}) with the initial data $(u_0, b_0) \in H^3(\mathbb{R}^3)$
	and $\nabla \cdot u_0 = \nabla \cdot b_0 = 0$, Then there exists a constant $\epsilon>0$ such that, if
	\begin{equation}\nonumber
	\|u_0\|_{H^3} + \|b_0\|_{H^3} \leq \epsilon,
	\end{equation}
	then  \eqref{mhd} has a unique global classical solution $(u, b)$ satisfying, for any $t>0$,
	$$
		\|u(t)\|_{H^3} + \|b(t)\|_{H^3} + \int_0^t \left(\|\na_h u\|_{H^3}^2 + \|\p_3 b\|^2_{H^3} + \|\p_1 b\|^2_{H^2}\right) \,d\tau \lesssim\;\epsilon.
	$$	
\end{theorem}

This new result constitutes an important contribution to the stability problem on the MHD equations. Prior to this stability result, we only know the stability of the background magnetic field for two cases, the ideal MHD equations and the MHD equations with kinematic dissipation and no magnetic diffusion. The nonlinear stability for the ideal MHD equations was established in several beautiful
papers \cite{BSS, CaiLei, HeXuYu, PanZhouZhu, WeiZ}. The stability problem for the MHD equations with no magnetic diffusion was first studied in \cite{LinZhang1}, which inspired many further investigations.  The stability has now been successfully established by several authors via different approaches (see, e.g., \cite{AZ, Deng, HuX, HuLin, HuWang, LinZhang1, PanZhouZhu, Ren, Ren2, Tan, WuWu, WuWuXu, TZhang}). To give a more complete view of current studies on the stability and the global regularity problems, we also mention some of the other exciting results in
\cite{CaoWuYuan, DongLiWu1, DongLiWu0, Fefferman1,Fefferman2,FNZ,HuangLi, JNW,JiuZhao2,WuMHD2018,
	Yam1,Yam2,Yam3,Yam4,Ye, YuanZhao, ZhouZhu} and the references therein.

\vskip .1in
A special consequence of Theorem \ref{good} and its proof is the stability or small data global well-posedness of the 3D Navier-Stokes equations with only horizontal dissipation. It is not clear if the stability for the 3D Navier-Stokes still holds if there is only one directional dissipation (say, in $x_1$ or $x_2$ direction, but not both). The 3D Navier-Stokes equations with full dissipation have small data global wellposedness while the 3D incompressible Euler equations
are ill-posed and have norm inflation in any Sobolev space $H^k$ or $C^k$ for any positive integer $k$ \cite{Bou, Bou1,Tarek}.

\vskip .1in
The proof of Theorem \ref{good} is not trivial. A natural starting point is to bound
$\|u(t)\|_{H^3} + \|b(t)\|_{H^3}$ via the energy estimates. However, due to the lack of the vertical dissipation and the horizontal magnetic diffusion, some of the nonlinear terms can not be controlled in terms of $\|u(t)\|_{H^3} + \|b(t)\|_{H^3}$ or the dissipative parts $\|\na_h u\|_{H^3}$ and $\|\p_3 b\|_{H^3}$. Consequently we are not able to obtain a closed differential inequality for
\begin{equation}\nonumber
E_0(t) = \sup_{ 0 \leq \tau \leq t} \Big\{ \|u(\tau)\|_{H^3}^2 + \|b(\tau)\|_{H^3}^2 \Big\} + 2 \int_0^t \|\na_h u(\tau)\|_{H^3}^2 + \|\partial_3 b(\tau)\|_{H^3}^2 \; d \tau.
\end{equation}
This forces us to include suitable extra terms in the energy estimates. We discover that the following term
\begin{equation}\nonumber
E_1(t) = \int_0^t \|\partial_1 b(\tau)\|_{H^2}^2  \; d \tau
\end{equation}
serves our purpose perfectly. All nonlinear terms involved in the estimates of $E_0(t)$ can be bounded in terms of $E_0(t)$ and $E_1(t)$. The selection of this term is based on the structure of (\ref{mhd}) and through trial and error. We remark that the process of estimating $E_0(t)$ involves many terms and is very lengthy. Even with the combination of $E_0(t)$ and $E_1(t)$, it is still very difficulty to directly bound some of the nonlinear terms. Two of the most difficult ones are
\beq \label{aa}
\int_{\mathbb{R}^3} \partial_1 u_1\, \partial_2^3 b_1\,  \partial_2^3 b_3 \; dx \quad
\mbox{and} \quad \int_{\mathbb{R}^3} \partial_1 u_3\, \partial_2^3 b_1 \, \partial_2^3 b_3 \; dx.
\eeq
It does not appear  possible to bound them directly in terms of $E_0(t)$ and $E_1(t)$.  Our strategy is to make use of the special structure of the equation for $b$ in
(\ref{mhd}) and replace $\p_1u_1$ and $\p_1 u_3$ in (\ref{aa}) via the equation of $b$,
\beq\label{bb}
\p_1 u = \p_t b + u\cdot\na b -\p_3^2 b - b\cdot\na u.
\eeq
Substituting (\ref{bb}) in (\ref{aa}) generates more terms, but fortunately all the resulting terms can be bounded suitably by $E_0(t)$ and $E_1(t)$.

\vskip .1in
In addition, in order to make most efficient usage of the anisotropic dissipation, we employ extensively the following anisotropic bounds in the estimates of the nonlinear terms. These anisotropic bounds are extremely powerful in the study
of global regularity and stability problems on partial differential equations
with only partial dissipation. Similar inequalities have previously been used in the investigation of partially dissipated 2D MHD systems and related equations (see, e.g., \cite{CaoReWu,CaoWu}).

\begin{lemma}\label{abound}
	The following estimates hold when the right-hand sides are all bounded.
	\begin{equation}\nonumber
	\begin{split}
	& \int_{\mathbb{R}^3} |f g h| dx \lesssim \|f\|_{L^2}^{\frac{1}{2}} \|\partial_1 f\|_{L^2}^{\frac{1}{2}}\|g\|_{L^2}^{\frac{1}{2}} \|\partial_2 g\|_{L^2}^{\frac{1}{2}}\|h\|_{L^2}^{\frac{1}{2}} \|\partial_3 h\|_{L^2}^{\frac{1}{2}},\\
	&\int_{\mathbb{R}^3} |f g h v| dx \lesssim  \|f\|_{L^2}^{\frac{1}{4}} \|\partial_1 f\|_{L^2}^{\frac{1}{4}} \|\partial_2 f\|_{L^2}^{\frac{1}{4}} \|\partial_1 \partial_2 f\|_{L^2}^{\frac{1}{4}} \|g\|_{L^2}^{\frac{1}{4}} \|\partial_1 g\|_{L^2}^{\frac{1}{4}} \|\partial_2 g\|_{L^2}^{\frac{1}{4}} \|\partial_1 \partial_2 g\|_{L^2}^{\frac{1}{4}} \\
	& \qquad \qquad \qquad \quad \cdot \|h\|_{L^2}^{\frac{1}{2}} \|\partial_3 h\|_{L^2}^{\frac{1}{2}} \|v\|_{L^2}^{\frac{1}{2}} \|\partial_3 v\|_{L^2}^{\frac{1}{2}},\\
	&\left(\int_{\mathbb{R}^3} |f g h|^2 dx\right)^\frac{1}{2} \lesssim  \|f\|_{L^2}^{\frac{1}{4}} \|\partial_1 f\|_{L^2}^{\frac{1}{4}} \|\partial_2 f\|_{L^2}^{\frac{1}{4}} \|\partial_1 \partial_2 f\|_{L^2}^{\frac{1}{4}} \|g\|_{L^2}^{\frac{1}{2}} \|\partial_3 g\|_{L^2}^{\frac{1}{2}} \|h\|_{H^2},\\
	&\int_{\mathbb{R}^3} |f g h| dx\lesssim \|f\|_{L^2}^{\frac{1}{4}} \|\partial_1 f\|_{L^2}^{\frac{1}{4}} \|\partial_2 f\|_{L^2}^{\frac{1}{4}} \|\partial_1 \partial_2 f\|_{L^2}^{\frac{1}{4}} \|g\|_{L^2}^{\frac{1}{2}} \|\partial_3 g\|_{L^2}^{\frac{1}{2}} \|h\|_{L^2}.
	\end{split}
	\end{equation}
\end{lemma}

\vskip .1in
Combining all aforementioned ingredients, we are able to drive the following energy inequalities
\begin{equation}\label{E0}
E_0(t) \lesssim E_0(0)+ E_0(0)^{\frac32} + E_0(t)^{\frac{3}{2}} + E_1(t)^{\frac{3}{2}} + E_0(t)^2 + E_1(t)^2
\end{equation}
and
\begin{equation}\label{E1}
E_1(t) \lesssim  E_0(0) + E_0(t) + E_0(t)^\frac{3}{2} + E_1(t)^\frac{3}{2}.
\end{equation}
These inequalities, combined with the bootstrapping argument, allow us to prove  Theorem \ref{good}.

\vskip .1in
The rest of this paper is divided into three sections. Section \ref{sec1} provides the proofs of Theorem \ref{good} and of Lemma \ref{abound}.  Section \ref{sec2} derives the energy inequality (\ref{E0}) while Section \ref{sec3} proves (\ref{E1}).

\vskip .3in
\section{Proofs of Theorem \ref{good} and Lemma \ref{abound}}
\label{sec1}
\setcounter {equation}{0}

This section proves Theorem \ref{good} and
Lemma \ref{abound}.

\vskip .1in
\begin{proof}[Proof of Theorem \ref{good}] We employ the bootstrapping argument (see, e.g., \cite[p.20]{Tao}). It follows from (\ref{E0}) and (\ref{E1}) that
\beno
E_0(t) + E_1(t) \lesssim E_0(0)+ E_0(0)^{\frac32} + E_0(t)^{\frac{3}{2}} + E_1(t)^{\frac{3}{2}} + E_0(t)^2 + E_1(t)^2
\eeno 	
or, for some pure constants $C_0$, $C_1$ and $C_2$,
\ben
E_0(t) + E_1(t) &\le& C_0 (E_0(0)+ E_0(0)^{\frac32}) + C_1 (E_0(t)^{\frac{3}{2}}+ E_1(t)^{\frac{3}{2}}) \notag\\
&&  + \,C_2 (E_0(t)^2 + E_1(t)^2).  \label{kkk}
\een	
To initiate the bootstrapping argument, we make the ansatz
\beq \label{ss}
E_0(t) + E_1(t) \le \min\left\{\frac{1}{16  C_1^2}, \,\frac1{4C_2}\right\}.
\eeq
We then show that (\ref{kkk}) allows us to conclude that $E_0(t) + E_1(t)$ actually
admits an even smaller bound by taking the initial $H^3$-norm $E_0(0)$ sufficiently small. In fact, when (\ref{ss}) holds, (\ref{kkk}) implies
\beno
E_0(t) + E_1(t) &\le& C_0 (E_0(0)+ E_0(0)^{\frac32}) + C_1\sqrt{E_0 +E_1}(E_0(t) + E_1(t)) \\
&& +\, C_2 (E_0(t) + E_1(t))(E_0(t) + E_1(t))\\
&\le& C_0 (E_0(0)+ E_0(0)^{\frac32}) + \frac12 (E_0(t) + E_1(t))
\eeno
or
\beq\label{ppp}
E_0(t) + E_1(t) \le 2 C_0 (E_0(0)+ E_0(0)^{\frac32}).
\eeq
Therefore, if we take $E_0(0)$ sufficiently small such that
\beq \label{mmm}
2 C_0 (E_0(0)+ E_0(0)^{\frac32}) < \min\left\{\frac{1}{16  C_1^2}, \,\frac1{4C_2}\right\},
\eeq
then $E_0(t) + E_1(t)$ actually
admits an smaller bound in (\ref{ppp}) than the one in the ansatz (\ref{ss}). The bootstrapping argument then assesses that (\ref{ppp}) holds for all time when
$E_0(0)$ obeys (\ref{mmm}).
This completes the proof.
\end{proof}

\vskip .1in
Next we prove Lemma \ref{abound}. A simple fact to be used in the proof is
the following version of Minkowski's inequality, for any $1\le q\le p\le \infty$,
$$
\| \|f\|_{L^q_y(\mathbb{R}^n)} \|_{L^p_x(\mathbb{R}^m)} \le \| \|f\|_{L^p_x(\mathbb{R}^m)} \|_{L^q_y(\mathbb{R}^n)},
$$
where  $f=f(x,y)$ with $x\in \mathbb{R}^m $ and $y\in \mathbb{R}^n$ is a measurable function on $\mathbb{R}^m\times \mathbb{R}^n$. A more general version of Minkowski's inequality and its proof can be found in \cite{Lieb}.

\begin{proof}[Proof of Lemma \ref{abound}] \quad The proof makes use of the following basic one-dimensional inequality, for $f\in H^1(\mathbb R)$,
	\beq \label{j9}
	\|f\|_{L^\infty(\mathbb R)} \le \sqrt{2}\, \|f\|^{\frac12}_{L^2(\mathbb R)} \, \|f'\|^{\frac12}_{L^2(\mathbb R)}.
	\eeq
	By H\"{o}lder's inequality and Minkowski's inequality,
	\beno
	\int_{\mathbb{R}^3} |f g h|\,dx &\le& \|f\|_{L^2_{x_3}L^2_{x_2}L^\infty_{x_1}}\, \|g\|_{L^2_{x_3} L^\infty_{x_2} L^2_{x_1}}\, \|h\|_{L^\infty_{x_3}L^2_{x_2}L^2_{x_1}} \\
	&\le&  \|f\|_{L^2_{x_3}L^2_{x_2}L^\infty_{x_1}}\, \|g\|_{L^2_{x_3}  L^2_{x_1}L^\infty_{x_2}}\, \|h\|_{L^2_{x_2}L^2_{x_1}L^\infty_{x_3}}
	\\
	&\le& 2^{\frac32} \, \left\|\|f\|_{L^2_{x_1}}^{\frac12} \, \|\p_1 f\|_{L^2_{x_1}}^{\frac12}\right\|_{L^2_{x_2x_3}}\, \left\|\|g\|_{L^2_{x_2}}^{\frac12} \, \|\p_2 g\|_{L^2_{x_2}}^{\frac12}\right\|_{L^2_{x_1x_3}}\,\\ && \times \left\|\|h\|_{L^2_{x_3}}^{\frac12} \, \|\p_3 h\|_{L^2_{x_3}}^{\frac12}\right\|_{L^2_{x_1x_2}}\,\\
	&\le&  2^{\frac32} \, \|f\|_{L^2}^{\frac12}\,\|\p_1f\|_{L^2}^{\frac12}\, \|g\|_{L^2}^{\frac12}\, \|\p_2 g\|_{L^2}^{\frac12}\, \|h\|_{L^2}^{\frac12}\, \|\p_3 h\|_{L^2}^{\frac12}.
	\eeno
	Here $\|f\|_{L^2_{x_3}L^2_{x_2}L^\infty_{x_1}}$ represents the $L^\infty$-norm in the $x_1$-variable, followed by the $L^2$-norm in $x_2$ and the $L^2$-norm in $x_3$. This finishes the proof of the first inequality. The proof of the second inequality is very similar. In fact, by H\"{o}lder's inequality and Minkowski's inequality,
	\beno
	\int |f g h v|\,dx \le \|f\|_{L^2_{x_3} L^\infty_{x_1} L^\infty_{x_2} }\, \|g\|_{L^2_{x_3}L^\infty_{x_1} L^\infty_{x_2} }\, \|h\|_{L^2_{x_1x_2}  L^\infty_{x_3}}\,  \|v\|_{L^2_{x_1x_2}  L^\infty_{x_3}}.
	\eeno
	By (\ref{j9}) and H\"{o}lder's inequality,
	\beno
	\|f\|_{L^2_{x_3} L^\infty_{x_1} L^\infty_{x_2}} &\le& 2^{\frac12}\,  \left\|\|f\|_{L^2_{x_2}}^{\frac12} \, \|\p_1 f\|_{L^2_{x_2}}^{\frac12}\right\|_{L^2_{x_3} L^\infty_{x_1} } \\
	&\le& 2^{\frac12}\, \left\|\|f\|_{L^\infty_{x_1}}\right\|^{\frac12}_{L^2_{x_2x_3}}
	\,\left\|\|\p_2 f\|_{L^\infty_{x_1}}\right\|^{\frac12}_{L^2_{x_2x_3}} \\
	&\le& 2^{\frac32}\,  \|f\|_{L^2}^{\frac14}\, \|\p_1 f\|_{L^2}^{\frac14}\,
	\|\p_2 f\|_{L^2}^{\frac14}\,  \|\p_1\p_2 f\|_{L^2}^{\frac14}.
	\eeno
	$\|g\|_{L^2_{x_3} L^\infty_{x_1}L^\infty_{x_2} }$ obeys a similar bound. $\|h\|_{L^2_{x_1x_2}  L^\infty_{x_3}}$ and   $\|v\|_{L^2_{x_1x_2}  L^\infty_{x_3}}$
	can be estimated as in the proof of the first inequality. Combining all these estimates leads to the desired second inequality in Lemma \ref{abound}. The other two
	inequalities are obtained similarly. This completes the
	proof of Lemma \ref{abound}.
\end{proof}

\vskip .3in
\section{Proof of (\ref{E0})}
\label{sec2}
\setcounter {equation}{0}

\vskip .1in
This section  proves (\ref{E0}), namely
$$
E_0(t) \lesssim  E_0(0)+ E_0(0)^{\frac32} + E_0(t)^{\frac{3}{2}} + E_1(t)^{\frac{3}{2}} + E_0(t)^2 + E_1(t)^2.
$$
The proof of this inequality is very lengthy and involves the estimates of many terms.

\begin{proof}[Proof of (\ref{E0})]\quad Due to the equivalence of $\|(u, b)\|_{H^3}$ with $\|(u, b)\|_{L^2} + \|(u, b)\|_{\dot{H}^3}$, it suffices to bound the $L^2$ and the homogeneous $\dot{H}^3$-norm of $(u,b)$. By a simple energy estimate and $\na\cdot u=\na\cdot b=0$, we find that
the $L^2$-norm  of $(u, b)$ obeys
$$
\|u(t)\|_{L^2}^2 + \|b(t)\|_{L^2}^2 + 2 \int_0^t \|\na_h u(\tau)\|_{L^2}^2 +
\|\p_3 b(\tau)\|_{L^2}^2\,d\tau  = \|u(0)\|_{L^2}^2 + \|b(0)\|_{L^2}^2.
$$
The rest of the proof focuses on the $\dot H^3$ norm. Applying $\partial_i^3$ $(i=1,2,3)$ to (\ref{mhd}) and then dotting by $(\partial_i^3 u, \partial_i^3 b)$, we find
\begin{equation}\label{E0:eq2}
\frac{1}{2} \frac{d}{dt} \sum_{i = 1}^3(\|\partial_i^3 u\|_{L^2}^2 + \|\partial_i^3 b\|_{L^2}^2) + \|\partial_i^3 \na_h u\|_{L^2}^2 + \|\partial_i^3 \partial_3 b\|_{L^2}^2
 = I_1 + I_2 + I_3 + I_4 + I_5,
\end{equation}
where
\begin{equation}\nonumber
\begin{split}
I_1 =& \sum_{i = 1}^3 \int_{\mathbb{R}^3} \partial_i^3 \partial_1 b \cdot \partial_i^3 u+ \partial_i^3 \partial_1 u \cdot\partial_i^3 b\; dx,\\
I_2 =& - \sum_{i = 1}^3\int_{\mathbb{R}^3} \partial_i^3 (u\cdot \nabla u)\cdot \partial_i^3 u \; dx,\\
I_3 =&  \sum_{i = 1}^3\int_{\mathbb{R}^3} [\partial_i^3 (b \cdot \nabla b) - b \cdot \nabla \partial_i^3 b]\cdot\partial_i^3 u \; dx,\\
I_4 =& - \sum_{i = 1}^3\int_{\mathbb{R}^3} \partial_i^3 (u \cdot \nabla b)\cdot \partial_i^3 b \; dx,\\
I_5 =& \sum_{i = 1}^3\int_{\mathbb{R}^3} [\partial_i^3 (b \cdot \nabla u) - b \cdot \nabla \partial_i^3 u]\cdot\partial_i^3 b \; dx.
\end{split}
\end{equation}
By integration by parts,  $I_1 = 0$.  To bound $I_2$, we decompose it into two pieces,
\begin{equation}\nonumber
\begin{split}
I_2 =& - \sum_{i = 1}^2\int_{\mathbb{R}^3}  \partial_i^3 (u\cdot \nabla u)\cdot \partial_i^3 u \; dx - \int_{\mathbb{R}^3} \partial_3^3 (u\cdot \nabla u)\cdot \partial_3^3 u \; dx \\
=&  \quad I_{2,1} + I_{2,2} .
\end{split}
\end{equation}
By H\"{o}lder's inequality,
\begin{equation}\label{jj}
\begin{split}
I_{2,1} \lesssim & \|\na_h (u \cdot \nabla u)\|_{\dot H^2} \|\na_h u\|_{\dot H^3}
\lesssim \|u\|_{H^3} \|\na_h u\|_{H^3}^2.
\end{split}
\end{equation}
By H\"{o}lder's inequality and Lemma \ref{abound},
\ben
I_{2,2} &=&- \int_{\mathbb{R}^3} \partial_3^3 (u_h \cdot \nabla_h u) \cdot \partial_3^3 u \;dx - \int_{\mathbb{R}^3} \partial_3^3 (u_3 \partial_3 u) \cdot \partial_3^3 u \;dx\nonumber\\
&= &- \sum_{k = 1}^3 \mathcal{C}_{3}^k\int_{\mathbb{R}^3} \partial_3^k u_h \cdot \nabla_h \partial_3^{3-k} u  \cdot \partial_3^3 u  + \partial_3^k u_3 \partial_3^{4-k} u \cdot \partial_3^3 u \;dx\nonumber\\
&\lesssim & \sum_{k = 1}^3 \|\partial_3^k u_h\|_{L^2}^\frac{1}{2} \|{\bf \partial_1} \partial_3^k u_h\|_{L^2}^\frac{1}{2} \|\nabla_h \partial_3^{3-k} u\|_{L^2}^\frac{1}{2}
\|{\bf \partial_3} \nabla_h \partial_3^{3-k} u\|_{L^2}^\frac{1}{2} \|\partial_3^3 u\|_{L^2}^\frac{1}{2} \|{\bf \partial_2} \partial_3^3 u\|_{L^2}^\frac{1}{2}\nonumber\\
&& + \; \|\partial_3^{k-1} \nabla_h \cdot u_h\|_{L^2}^\frac{1}{2}\|{\bf \partial_3} \partial_3^{k-1} \nabla_h \cdot u_h\|_{L^2}^\frac{1}{2} \|\partial_3^{4-k} u\|_{L^2}^\frac{1}{2}
\|{\bf \partial_1} \partial_3^{4-k} u\|_{L^2}^\frac{1}{2} \notag\\
&& \qquad \times \| \partial_3^3 u\|_{L^2}^\frac{1}{2} \|{\bf \partial_2} \partial_3^3 u\|_{L^2}^\frac{1}{2}\nonumber\\
&\lesssim & \|u\|_{H^3} \|\na_h u\|_{H^3}^2. \label{jj1}
\een
Now we turn to the next term  $I_3$,
\begin{equation}\nonumber
\begin{split}
I_3 =& \sum_{i = 1}^3 \sum_{k = 1}^3 \mathcal{C}_{3}^k\int_{\mathbb{R}^3} \partial_i^k b \cdot \nabla \partial_i^{3-k} b \cdot\partial_i^3 u \; dx = I_{3,1} + I_{3,2}+I_{3,3}.
\end{split}
\end{equation}
By H\"{o}lder's and Sobolev's inequalities,
\ben
I_{3,1} &\le & \Big( \|\partial_1 b\|_{L^3}  \|\nabla \partial_1^2 b\|_{L^2} + \|\partial_1^2 b\|_{L^3}  \|\nabla \partial_1 b\|_{L^2} + \|\partial_1^3 b\|_{L^2}  \|\nabla  b\|_{L^3} \Big)  \|\partial_1^3 u \|_{L^6} \notag\\
&\lesssim& \|b\|_{H^3} \|\partial_1 b\|_{H^2} \|\partial_1 u \|_{H^3}. \label{jj2}
\een
By Lemma \ref{abound},
\ben
I_{3,2} &\lesssim& \sum_{k = 1}^2 \|\partial_2^k b\|_{L^2}^\frac{1}{2} \|{\bf \partial_1} \partial_2^k b\|_{L^2}^\frac{1}{2}  \|\nabla \partial_2^{3-k} b\|_{L^2}^\frac{1}{2}
\|{\bf \partial_3} \nabla \partial_2^{3-k} b\|_{L^2}^\frac{1}{2} \|\partial_2^3 u \|_{L^2}^\frac{1}{2}\|{\bf \partial_2} \partial_2^3 u \|_{L^2}^\frac{1}{2} \notag\\
&& + \|\partial_2^3 b\|_{L^2}^\frac{1}{2} \|{\bf \partial_3} \partial_2^3 b\|_{L^2}^\frac{1}{2}  \|\nabla  b\|_{L^2}^\frac{1}{2}
\|{\bf \partial_1} \nabla  b\|_{L^2}^\frac{1}{2} \|\partial_2^3 u \|_{L^2}^\frac{1}{2}\|{\bf \partial_2} \partial_2^3 u \|_{L^2}^\frac{1}{2} \notag\\
&\lesssim& \|b\|_{H^3} \|\partial_1 b\|_{H^2}^\frac{1}{2} \|\partial_3 b\|_{H^3}^\frac{1}{2} \|\partial_2 u \|_{H^3} \label{jj3}
\een
and
\ben
I_{3,3} &\lesssim& \sum_{k = 1}^2 \|\partial_3^k b\|_{L^2}^\frac{1}{2} \|{\bf \partial_1} \partial_3^k b\|_{L^2}^\frac{1}{2}  \|\nabla \partial_3^{3-k} b\|_{L^2}^\frac{1}{2}
\|{\bf \partial_3} \nabla \partial_3^{3-k} b\|_{L^2}^\frac{1}{2} \|\partial_3^3 u \|_{L^2}^\frac{1}{2}\|{\bf \partial_2} \partial_3^3 u \|_{L^2}^\frac{1}{2} \notag\\
&& + \|\partial_3^3 b\|_{L^2}^\frac{1}{2} \|{\bf \partial_3} \partial_3^3 b\|_{L^2}^\frac{1}{2}  \|\nabla  b\|_{L^2}^\frac{1}{2}
\|{\bf \partial_1} \nabla  b\|_{L^2}^\frac{1}{2} \|\partial_3^3 u \|_{L^2}^\frac{1}{2}\|{\bf \partial_2} \partial_3^3 u \|_{L^2}^\frac{1}{2} \notag\\
&\lesssim& \|b\|_{H^3}^{\frac12}\, \|u\|_{H^3}^{\frac12}\,\|\partial_1 b\|_{H^2}^\frac{1}{2}\, \|\partial_3 b\|_{H^3} \|\partial_2 u \|_{H^3}^{\frac12}. \label{jj4}
\een
The next term $I_4$ is naturally split into three parts,
\begin{equation}\nonumber
I_{4} =  - \sum_{i = 1}^3\int_{\mathbb{R}^3} \partial_i^3 (u \cdot \nabla b) \partial_i^3 b \; dx = I_{4,1}+I_{4,2}+I_{4,3}.
\end{equation}
$I_{4,1}$ and $I_{4,3}$ involves the favorable partial derivatives in $x_1$ and $x_3$, respectively, and their handling is not very difficult. In contrast, $I_{4,2}$
has partials in terms of $x_2$ and the control of $I_{4,2}$ is extremely delicate. By Lemma \ref{abound},
\begin{equation}\nonumber
\begin{split}
I_{4, 1} =&  - \sum_{k = 1}^3 \mathcal{C}_{3}^k\int_{\mathbb{R}^3} \partial_1^k u \cdot \nabla \partial_1^{3-k} b\cdot \partial_1^3 b \; dx,\\
\lesssim& \; \|\partial_1 u\|_{L^\infty} \|\nabla \partial_1^2 b\|_{L^2} \|\partial_1^3 b\|_{L^2} + \sum_{k = 2}^3 \|\partial_1^k u\|_{L^6} \|\nabla \partial_1^{3-k} b\|_{L^3} \|\partial_1^3 b\|_{L^2}\\
\lesssim& \; \|b\|_{H^3} \|\partial_1 u\|_{H^3} \|\partial_1 b\|_{H^2}
\end{split}
\end{equation}
and
\ben
I_{4, 3} &=&  - \sum_{k = 1}^3 \mathcal{C}_{3}^k\int_{\mathbb{R}^3} \partial_3^k u \cdot \nabla \partial_3^{3-k} b \cdot\partial_3^3 b \; dx, \notag\\
&\lesssim& \sum_{k = 1}^3\|\partial_3^k u\|_{L^2}^\frac{1}{2} \|{\bf \partial_2} \partial_3^k u\|_{L^2}^\frac{1}{2} \|\nabla \partial_3^{3-k} b\|_{L^2}^\frac{1}{2}
\|{\bf \partial_3} \nabla \partial_3^{3-k} b\|_{L^2}^\frac{1}{2}\|\partial_3^3 b\|_{L^2}^\frac{1}{2} \|{\bf \partial_1} \partial_3^3 b\|_{L^2}^\frac{1}{2} \notag\\
&\lesssim & \|u\|_{H^3}^\frac{1}{2}\|b\|_{H^3}^\frac{1}{2}\|\partial_2 u\|_{H^3}^\frac{1}{2} \|\partial_3 b\|_{H^3}^\frac{3}{2}. \label{pp1}
\een
We now turn to $I_{4,2}$, one of the most difficult terms. We further decompose it into three terms,
\begin{equation}\nonumber
\begin{split}
I_{4, 2} =& - \int_{\mathbb{R}^3} \big( \partial_2^3 (u_1 \partial_1 b) + \partial_2^3 (u_2 \partial_2 b) + \partial_2^3 (u_3 \partial_3 b) \big) \cdot\partial_2^3 b \; dx \\
= & \; I_{4,2,1}+I_{4,2,2}+I_{4,2,3}.
\end{split}
\end{equation}
By Lemma \ref{abound},
\begin{equation}\nonumber
\begin{split}
I_{4, 2, 1} =& - \sum_{k = 1}^3 \mathcal{C}_{3}^k \int_{\mathbb{R}^3}   \partial_2^k u_1 \partial_1 \partial_2^{3-k}b \cdot\partial_2^3 b \; dx \\
\lesssim & \sum_{k = 2}^3 \|  \partial_2^k u_1 \|_{L^6}\|\partial_1 \partial_2^{3-k}b\|_{L^3} \|\partial_2^3 b\|_{L^2} + \|\partial_2 u_1\|_{L^\infty} \|\partial_1 \partial_2^2 b\|_{L^2} \|\partial_2^3 b\|_{L^2}\\
\lesssim & \; \|b\|_{H^3} \|\partial_2 u_1\|_{H^3} \|\partial_1 b\|_{H^2}
\end{split}
\end{equation}
and
\begin{equation}\nonumber
\begin{split}
I_{4, 2, 3} =& - \sum_{k = 1}^3 \mathcal{C}_{3}^k \int_{\mathbb{R}^3}   \partial_2^k u_3 \partial_3 \partial_2^{3-k}b \cdot\partial_2^3 b \; dx \\
\lesssim & \sum_{k = 1}^3 \|  \partial_2^k u_3 \|_{L^6}\|\partial_3 \partial_2^{3-k}b\|_{L^3} \|\partial_2^3 b\|_{L^2} \\
\lesssim & \; \|b\|_{H^3} \|\partial_2 u_3\|_{H^3} \|\partial_3 b\|_{H^3}.
\end{split}
\end{equation}
$I_{4,2,2}$ is much more challenging and we further break it down,
\begin{equation}\nonumber
\begin{split}
I_{4, 2, 2} =&- \sum_{k = 2}^3 \mathcal{C}_{3}^k \int_{\mathbb{R}^3}  \partial_2^k u_2 \partial_2^{4-k} b\cdot \partial_2^3 b \; dx - 3\int_{\mathbb{R}^3}  \partial_2 u_2 \partial_2^3 b \cdot\partial_2^3 b \; dx\\
= & - \sum_{k = 2}^3 \mathcal{C}_{3}^k \int_{\mathbb{R}^3}  \partial_2^k u_2 \partial_2^{4-k} b \cdot\partial_2^3 b \; dx + 3\int_{\mathbb{R}^3}  \partial_3 u_3 \partial_2^3 b \cdot\partial_2^3 b \; dx \\
&+ 3\int_{\mathbb{R}^3}  \partial_1 u_1 \partial_2^3 b \cdot\partial_2^3 b \; dx  = I_{4,2,2,1} + I_{4,2,2,2} + I_{4,2,2,3}.
\end{split}
\end{equation}
By Lemma \ref{abound},
\begin{equation}\nonumber
\begin{split}
I_{4, 2, 2, 1} \lesssim & \sum_{k = 2}^3 \|\partial_2^k u_2\|_{L^2}^\frac{1}{2} \|{\bf \partial_2}\partial_2^k u_2\|_{L^2}^\frac{1}{2} \|\partial_2^{4-k} b\|_{L^2}^\frac{1}{2} \|{\bf \partial_1} \partial_2^{4-k} b\|_{L^2}^\frac{1}{2} \|\partial_2^3 b\|_{L^2}^\frac{1}{2}\|{\bf \partial_3}\partial_2^3 b\|_{L^2}^\frac{1}{2}\\
 \lesssim & \; \|b\|_{H^3} \|\partial_2 u_2\|_{H^3} \|\partial_1 b\|_{H^2}^\frac{1}{2} \|\partial_3 b\|_{H^3}^\frac{1}{2}.
\end{split}
\end{equation}
By integration by parts and Lemma \ref{abound},
\begin{equation}\nonumber
\begin{split}
I_{4, 2, 2, 2} \lesssim &\; \|u_3\|_{L^2}^\frac{1}{4}\|{\bf \partial_1} u_3\|_{L^2}^\frac{1}{4}\|{\bf \partial_2} u_3\|_{L^2}^\frac{1}{4}\|{\bf \partial_1 \partial_2} u_3\|_{L^2}^\frac{1}{4} \| \partial_2^3 b \|_{L^2}^\frac{1}{2}\|{\bf \partial_3} \partial_2^3 b \|_{L^2}^\frac{1}{2}\|\partial_3 \partial_2^3 b \|_{L^2}\\
\lesssim &\; \|b\|_{H^3}^\frac{1}{2} \|u_3\|_{H^3}^\frac{1}{2}
\|\partial_3 b\|_{H^3}^\frac{3}{2} \|\partial_2 u_3\|_{H^3}^\frac{1}{2}.
\end{split}
\end{equation}
$I_{4,2,2,3}$ can not be directly estimated to yield a suitable bound. If we attempt to directly apply Lemma \ref{abound} as follows,
$$ \int_{\mathbb{R}^3} \partial_1 u_1 \partial_2^3 b  \partial_2^3 b \; dx \leq \|\partial_1 u_1\|_{L^2}^{\frac{1}{2}} \|\partial_1 \partial_2u_1\|_{L^2}^{\frac{1}{2}}\|\partial_2^3 b\|_{L^2}^{\frac{1}{2}}\|\partial_3\partial_2^3 b\|_{L^2}^{\frac{1}{2}}\|\partial_2^3 b\|_{L^2}^{\frac{1}{2}}\|\partial_1\partial_2^3 b\|_{L^2}^{\frac{1}{2}},$$
which involves $\|\partial_1 b\|_{H^3}$ and the differential inequality would not be closed. Fortunately the equations in (\ref{mhd}) has a special structure. The equation of $b$ allows us to replace $\p_1 u_1$ via
$$
\p_1 u = \p_t u + u\cdot\na b -\p_3^2 b - b\cdot\na u
$$
and bring us the hope of controlling $I_{4, 2, 2, 3}$ suitably. We write $I_{4, 2, 2, 3}$ as
\begin{equation}\nonumber
\begin{split}
I_{4, 2, 2, 3} =& \; 3 \int_{\mathbb{R}^3} \partial_1 u_1 \partial_2^3 b \cdot \partial_2^3 b \; dx \\
= &  \; 3 \int_{\mathbb{R}^3} [\partial_t b_1 + u\cdot \nabla b_1 - \partial_3^2 b_1 - b \cdot \nabla u_1] |\partial_2^3 b|^2 \; dx\\
= & \; J_1 + J_2 + J_3 + J_4.
\end{split}
\end{equation}
$J_2$, $J_3$ and $J_4$ are relatively easier to deal with. By Lemma \ref{abound},
\beno
J_2 &\lesssim & \|u\|_{L^2}^\frac{1}{4}\|{\bf \partial_1 } u\|_{L^2}^\frac{1}{4} \|{\bf \partial_2} u\|_{L^2}^\frac{1}{4} \|{\bf \partial_1 \partial_2} u\|_{L^2}^\frac{1}{4}
\|\nabla b\|_{L^2}^\frac{1}{4}\|{\bf \partial_1 } \nabla b\|_{L^2}^\frac{1}{4} \notag\\
&& \times \|{\bf \partial_2} \nabla b\|_{L^2}^\frac{1}{4} \|{\bf \partial_1 \partial_2} \nabla b\|_{L^2}^\frac{1}{4}  \|\partial_2^3 b\|_{L^2} \|{\bf \partial_3} \partial_2^3 b\|_{L^2} \notag\\
&\lesssim & \|u\|_{H^3}^\frac{1}{2}\|b\|_{H^3}^\frac{3}{2} \|\na_h u\|_{H^3}^\frac{1}{2} \|\partial_1 b\|_{H^2}^\frac{1}{2} \|\partial_3 b\|_{H^3}.
\notag
\eeno
By H\"{o}lder's inequality,
\begin{equation}\nonumber
\begin{split}
J_3 \lesssim \|\partial_3 b_1\|_{L^\infty} \|\partial_3 \partial_2^3 b\|_{L^2}  \|\partial_2^3 b\|_{L^2}\lesssim \|b\|_{H^3} \|\partial_3 b\|_{H^3}^2.
\end{split}
\end{equation}
Again by Lemma \ref{abound},
\ben
J_4 &\lesssim & \;\|\nabla u_1\|_{L^2}^\frac{1}{4}\|{\bf \partial_1 } \nabla u_1\|_{L^2}^\frac{1}{4} \|{\bf \partial_2} \nabla u_1\|_{L^2}^\frac{1}{4} \|{\bf \partial_1 \partial_2} \nabla u_1\|_{L^2}^\frac{1}{4} \notag\\
&&\times \|b\|_{L^2}^\frac{1}{4}\|{\bf \partial_1 } b\|_{L^2}^\frac{1}{4} \|{\bf \partial_2} b\|_{L^2}^\frac{1}{4} \|{\bf \partial_1 \partial_2} b \|_{L^2}^\frac{1}{4}\,
\|\partial_2^3 b\|_{L^2} \|{\bf \partial_3} \partial_2^3 b\|_{L^2} \notag\\
&\lesssim & \; \|u\|_{H^3}^\frac{1}{2} \|b\|_{H^3}^\frac{3}{2} \|\na_h u\|_{H^3}^\frac{1}{2} \|\partial_1 b\|_{H^2}^\frac{1}{2} \|\partial_3 b\|_{H^3}. \label{pp2}
\een
To deal with $J_1$, we rewrite it as
\ben
J_1  = 3 \frac{d}{dt}\int_{\mathbb{R}^3} b_1\, |\partial_2 ^3 b|^2 \;dx - 3 \int_{\mathbb{R}^3} b_1 \partial_t |\partial_2 ^3 b|^2 \;dx = J_{1,1} + J_{1,2}. \label{pp3}
\een
To estimate $J_{1,2}$, we use the equation of $b$ in (\ref{mhd}) again to write is as
\begin{equation}\nonumber
\begin{split}
J_{1,2}  =&  -6 \int_{\mathbb{R}^3} b_1 \partial_2^3 b \cdot\big( - \partial_2^3(u\cdot \nabla b) - \partial_2^3 \partial_3^2 b + \partial_2^3 (b \cdot \nabla u) + \partial_2^3 \partial_1 u \big)\;dx \\
= & \; J_{1,2,1}+J_{1,2,2} + J_{1,2,3}+J_{1,2,4}.
\end{split}
\end{equation}
By integrating by parts and applying Lemma \ref{abound}, we have
\begin{equation}\nonumber
\begin{split}
J_{1,2,1}  =&  \; 6 \sum_{k = 1}^3 \int_{\mathbb{R}^3} \mathcal{C}_{3}^k b_1 \partial_2^3 b \partial_2^k u\cdot \nabla \partial_2^{3-k} b\; dx - 3\int_{\mathbb{R}^3} u \cdot \nabla b_1 |\partial_2^3 b|^2 \; dx\\
\lesssim &\; \|b_1\|_{L^2}^\frac{1}{4} \|{\bf \partial_1} b_1\|_{L^2}^\frac{1}{4} \|{\bf \partial_2} b_1\|_{L^2}^\frac{1}{4}
\|{\bf \partial_1 \partial_2} b_1\|_{L^2}^\frac{1}{4}\\
&\; \times   \|\partial_2 u\|_{L^2}^\frac{1}{4} \|{\bf \partial_1} \partial_2 u\|_{L^2}^\frac{1}{4} \|{\bf \partial_2} \partial_2 u\|_{L^2}^\frac{1}{4}
\|{\bf \partial_1 \partial_2} \partial_2 u\|_{L^2}^\frac{1}{4} \\
&\; \times  \|\partial_2^3 b\|_{L^2}^\frac{1}{2}\|{\bf \partial_3} \partial_2^3 b \|_{L^2}^\frac{1}{2}\, \|\nabla \partial_2^2 b \|_{L^2}^\frac{1}{2}\|{\bf \partial_3} \nabla \partial_2^2 b \|_{L^2}^\frac{1}{2}\\
&+\sum_{k = 2}^3 \|b_1\|_{L^\infty}\|\partial_2^3 b\|_{L^2}^\frac{1}{2} \|{\bf \partial_3}\partial_2^3 b\|_{L^2}^\frac{1}{2} \|\partial_2^k u\|_{L^2}^\frac{1}{2} \|{\bf \partial_2} \partial_2^k u\|_{L^2}^\frac{1}{2} \\
&\;  \qquad \times\|\nabla \partial_2^{3-k} b\|_{L^2}^\frac{1}{2} \|{\bf \partial_1} \nabla \partial_2^{3-k} b\|_{L^2}^\frac{1}{2} + |J_2|\\
\lesssim & \; \|u\|_{H^3}^\frac{1}{2} \|b\|_{H^3}^\frac{3}{2} \|\na_h u\|_{H^3}^\frac{1}{2} \|\partial_1 b\|_{H^2}^\frac{1}{2}\|\partial_3 b\|_{H^3} +
\|b\|_{H^3}^2 \|\na_h u\|_{H^3} \|\partial_1 b\|_{H^2}^\frac{1}{2} \|\partial_3 b\|_{H^3}^\frac{1}{2}.
\end{split}
\end{equation}
By integration by parts and H\"{o}lder's inequality,
\begin{equation}\nonumber
\begin{split}
J_{1,2,2}  =&  \; 3 \int_{\mathbb{R}^3} \partial_3^2 b_1 |\partial_2^3 b|^2 \;dx - 6 \int_{\mathbb{R}^3}  b_1 |\partial_3 \partial_2^3 b|^2 \;dx \\
\lesssim & \; |J_3| + \|b_1\|_{L^\infty} \|\partial_3 b\|_{H^3}^2 \\
\lesssim &  \; \|b\|_{H^3} \|\partial_3 b\|_{H^3}^2.
\end{split}
\end{equation}
By Lemma \ref{abound},
\begin{equation}\nonumber
\begin{split}
J_{1,2,3}  =&  -6 \sum_{k = 0}^3 \int_{\mathbb{R}^3} \mathcal{C}_{3}^k b_1 \partial_2^3 b \cdot \partial_2^k b\cdot \nabla \partial_2^{3-k} u\; dx \\
= & -6 \sum_{k = 1}^2 \int_{\mathbb{R}^3} \mathcal{C}_{3}^k b_1 \partial_2^3 b \partial_2^k b\cdot \nabla \partial_2^{3-k} u\; dx   -6  \int_{\mathbb{R}^3} b_1 \partial_2^3 b \partial_2^3 b\cdot \nabla u\; dx\\
&\;   -6  \int_{\mathbb{R}^3}b_1 \partial_2^3 b  b\cdot \nabla \partial_2^3 u\; dx \\
\lesssim &\; \sum_{k = 1}^2 \|b_1\|_{L^\infty} \|\partial_2^3 b\|_{L^2}^\frac{1}{2} \|{\bf \partial_3} \partial_2^3 b\|_{L^2}^\frac{1}{2} \|\partial_2^k b\|_{L^2}^\frac{1}{2}
\|{\bf \partial_1} \partial_2^k b\|_{L^2}^\frac{1}{2} \|\nabla \partial_2^{3-k} u\|_{L^2}^\frac{1}{2} \|{\bf \partial_2}\nabla \partial_2^{3-k} u\|_{L^2}^\frac{1}{2} \\
& + \|b_1\|_{L^2}^\frac{1}{4}\|{\bf \partial_1}b_1\|_{L^2}^\frac{1}{4}\|{\bf \partial_2}b_1\|_{L^2}^\frac{1}{4}
\|{\bf \partial_1 \partial_2}b_1\|_{L^2}^\frac{1}{4}\,\\
&\quad \times  \|\nabla u\|_{L^2}^\frac{1}{4}\|{\bf \partial_1}\nabla u \|_{L^2}^\frac{1}{4}\|{\bf \partial_2} \nabla u\|_{L^2}^\frac{1}{4}
\|{\bf \partial_1 \partial_2}\nabla u\|_{L^2}^\frac{1}{4}
\|\partial_2^3 b\|_{L^2} \|{\bf \partial_3 } \partial_2^3 b\|_{L^2}\\
& + \|b_1 \partial_2^3 b b\|_{L^2} \|\partial_2 u\|_{H^3}\\
\lesssim & \; \|b\|_{H^3}^2 \|\partial_2 u\|_{H^3} \|\partial_1 b\|_{H^2}^\frac{1}{2} \|\partial_3 b\|_{H^3}^\frac{1}{2} + \|u\|_{H^3}^\frac{1}{2} \|b\|_{H^3}^\frac{3}{2}
\|\partial_h u\|_{H^3}^\frac{1}{2} \|\partial_1 b\|_{H^2}^\frac{1}{2} \|\partial_3 b\|_{H^3}\\
& + \|b_1\|_{L^\infty} \|b\|_{L^2}^\frac{1}{4}\|{\bf \partial_1 }b\|_{L^2}^\frac{1}{4}\|{\bf \partial_2 }b\|_{L^2}^\frac{1}{4}\|{\bf \partial_1 \partial_2}b\|_{L^2}^\frac{1}{4}
\|\partial_2^3 b\|_{L^2}^\frac{1}{2}\|{\bf \partial_3} \partial_2^3 b\|_{L^2}^\frac{1}{2} \|b\|_{H^2} \|\partial_2 u\|_{H^3}\\
\lesssim & \; \|b\|_{H^3}^2 \|\partial_2 u\|_{H^3} \|\partial_1 b\|_{H^2}^\frac{1}{2} \|\partial_3 b\|_{H^3}^\frac{1}{2} + \|u\|_{H^3}^\frac{1}{2} \|b\|_{H^3}^\frac{3}{2}
\|\partial_h u\|_{H^3}^\frac{1}{2} \|\partial_1 b\|_{H^2}^\frac{1}{2} \|\partial_3 b\|_{H^3}\\
& + \|b\|_{H^3}^2 \|\partial_2 u\|_{H^3} \|\partial_1 b\|_{H^2}^\frac{1}{2} \|\partial_3 b\|_{H^3}^\frac{1}{2}.
\end{split}
\end{equation}
The last term $J_{1,2,4}$ can also be bounded via Lemma \ref{abound},
\begin{equation}\nonumber
\begin{split}
J_{1,2,4}  \lesssim & \; \|b_1\|_{L^2}^\frac{1}{4} \|{\bf \partial_1} b_1\|_{L^2}^\frac{1}{4} \|{\bf \partial_2} b_1\|_{L^2}^\frac{1}{4} \|{\bf \partial_1 \partial_2} b_1\|_{L^2}^\frac{1}{4} \|\partial_2^3 b \|_{L^2}^\frac{1}{2} \|{\bf \partial_3} \partial_2^3 b \|_{L^2}^\frac{1}{2} \|\partial_2^3 \partial_1 u\|_{L^2} \\
\lesssim & \; \|b\|_{H^3} \|\partial_1 u\|_{H^3} \|\partial_1 b\|_{H^2}^\frac{1}{2} \|\partial_3 b\|_{H^3}^\frac{1}{2}.
\end{split}
\end{equation}
It remains to estimate $I_5$,
\begin{equation}\nonumber
I_5 = \sum_{i = 1}^3\int_{\mathbb{R}^3} \Big( \partial_i^3 (b \cdot \nabla u) - b \cdot \nabla \partial_i^3 u \Big)\cdot  \partial_i^3 b \; dx = I_{5,1} +I_{5,2}+I_{5,3}.
\end{equation}
By H\"{o}lder's and Sobolev's inequalities,
\begin{equation}\nonumber
\begin{split}
I_{5,1} = &\sum_{k = 1}^3 \mathcal{C}_{3}^k \int_{\mathbb{R}^3} \partial_1^k b \cdot \nabla \partial_1^{3-k} u  \cdot\partial_1^3 b \; dx\\
\lesssim & \sum_{k = 1}^2 \|\partial_1^k b\|_{L^3} \|\nabla \partial_1^{3-k} u\|_{L^6} \| \partial_1^3 b\|_{L^2} +
\|\partial_1^3 b\|_{L^2}^2 \|\nabla u\|_{L^\infty}\\
\lesssim & \; \|b\|_{H^3} \|\partial_1 u\|_{H^3} \|\partial_1 b\|_{H^2} + \|u\|_{H^3} \|\partial_1 b\|_{H^2}^2
\end{split}
\end{equation}
and
\begin{equation}\nonumber
\begin{split}
I_{5,3} = &\sum_{k = 1}^3 \mathcal{C}_{3}^k \int_{\mathbb{R}^3} \partial_3^k b \cdot \nabla \partial_3^{3-k} u \cdot \partial_3^3 b \; dx\\
\lesssim & \sum_{k = 1}^3 \|\partial_3^k b\|_{L^3} \|\nabla \partial_3^{3-k} u\|_{L^2}  \|\partial_3^3 b\|_{L^6}\\
\lesssim & \; \|u\|_{H^3} \|\partial_3 b\|_{H^3}^2.
\end{split}
\end{equation}
The difficult term is $I_{5,2}$, which is further decomposed into
\begin{equation}\nonumber
\begin{split}
I_{5,2} = & \int_{\mathbb{R}^3} \Big[\partial_2^3 (b_1 \partial_1 u)+ \partial_2^3 (b_2 \partial_2 u) + \partial_2^3 (b_3 \partial_3 u)\Big]\cdot\partial_2^3 b \; dx \\
= &\; I_{5,2,1}+I_{5,2,2}+I_{5,2,3}.
\end{split}
\end{equation}
The last two terms $I_{5,2,2}$ and $I_{5,2,3}$ can be directly bounded.  By $\na\cdot b=0$,
\begin{equation}\nonumber
\begin{split}
I_{5,2,2} = & \sum_{k = 1}^3 \mathcal{C}_{3}^k \int_{\mathbb{R}^3} \partial_2^k b_2 \partial_2 \partial_2^{3-k} u\cdot\partial_2^3 b \; dx\\
\lesssim & \sum_{k = 1}^2 \|\partial_2^{k-1} (\partial_1 b_1 + \partial_3 b_3)\|_{L^3} \|\partial_2 \partial_2^{3-k} u\|_{L^6}\|\partial_2^3 b\|_{L^2}\\
&\;
+ \|\partial_2^2 (\partial_1 b_1 + \partial_3 b_3)\|_{L^2} \|\partial_2  u\|_{L^\infty}\|\partial_2^3 b\|_{L^2}\\
\lesssim & \; \|b\|_{H^3} \|\partial_2 u\|_{H^3} (\|\partial_1 b\|_{H^2} + \|\partial_3 b\|_{H^3}).
\end{split}
\end{equation}
By integration by parts and Lemma \ref{abound},
\begin{equation}\nonumber
\begin{split}
I_{5,2,3} = & \sum_{k = 1}^3 \mathcal{C}_{3}^k \int_{\mathbb{R}^3} \partial_2^k b_3 \partial_3 \partial_2^{3-k} u\cdot\partial_2^3 b \; dx\\
= & -\sum_{k = 1}^3 \mathcal{C}_{3}^k \int_{\mathbb{R}^3} \partial_2^k \partial_3 b_3 \partial_2^{3-k} u\cdot\partial_2^3 b \; dx  -
\sum_{k = 1}^3 \mathcal{C}_{3}^k \int_{\mathbb{R}^3} \partial_2^k b_3 \partial_2^{3-k} u\partial_2^3 \cdot\partial_3 b \; dx\\
\lesssim & \sum_{k = 1}^3\|\partial_2^k \partial_3 b_3\|_{L^2} \|\partial_2^{3-k} u\|_{L^2}^\frac{1}{4}\|{\bf \partial_1} \partial_2^{3-k} u\|_{L^2}^\frac{1}{4}
\|{\bf \partial_2} \partial_2^{3-k} u\|_{L^2}^\frac{1}{4} \\
&\, \quad \times \|{\bf \partial_1 \partial_2} \partial_2^{3-k} u\|_{L^2}^\frac{1}{4}\|\partial_2^3 b\|_{L^2}^\frac{1}{2}
\|{\bf \partial_3}\partial_2^3 b\|_{L^2}^\frac{1}{2}\\
& + \|\partial_2^k b_3\|_{L^2}^\frac{1}{2}\|{\bf \partial_3} \partial_2^k b_3\|_{L^2}^\frac{1}{2} \|\partial_2^{3-k} u\|_{L^2}^\frac{1}{4}\|{\bf \partial_1} \partial_2^{3-k} u\|_{L^2}^\frac{1}{4}
\|{\bf \partial_2} \partial_2^{3-k} u\|_{L^2}^\frac{1}{4}\\
& \quad \times \|{\bf \partial_1 \partial_2} \partial_2^{3-k} u\|_{L^2}^\frac{1}{4}\|\partial_2^3 \partial_3 b\|_{L^2}\\
\lesssim & \; \|u\|_{H^3}^\frac{1}{2}\|b\|_{H^3}^\frac{1}{2}\|\na_h u\|_{H^3}^\frac{1}{2} \|\partial_3 b\|_{H^3}^\frac{3}{2}.
\end{split}
\end{equation}
The estimate for $I_{5,2,1}$ is much more complex. We further break it down,
\begin{equation}\nonumber
\begin{split}
I_{5,2,1} = & \sum_{k = 1}^3 \mathcal{C}_{3}^k \int_{\mathbb{R}^3} \partial_2^k b_1 \partial_1 \partial_2^{3-k} u\cdot\partial_2^3 b \; dx\\
= & \sum_{k = 1}^3 \mathcal{C}_{3}^k \int_{\mathbb{R}^3} \partial_2^k b_1 \partial_1 \partial_2^{3-k} u_1\partial_2^3 b_1
+ \partial_2^k b_1 \partial_1 \partial_2^{3-k} u_2\partial_2^3 b_2 + \partial_2^k b_1 \partial_1 \partial_2^{3-k} u_3\partial_2^3 b_3\; dx \\
=& \; I_{5,2,1,1} + I_{5,2,1,2} +I_{5,2,1,3} .
\end{split}
\end{equation}
We estimate  $I_{5,2,1,1}$ and $I_{5,2,1,2}$ directly. By Lemma \ref{abound},
\begin{equation}\nonumber
\begin{split}
I_{5,2,1,1}
= & \sum_{k = 1}^2 \mathcal{C}_{3}^k \int_{\mathbb{R}^3} \partial_2^k b_1 \partial_1 \partial_2^{3-k} u_1\partial_2^3 b_1\; dx
+ \int_{\mathbb{R}^3} \partial_2^3 b_1 \partial_1 u_1\partial_2^3 b_1\; dx\\
\lesssim & \sum_{k = 1}^2 \|\partial_2^k b_1\|_{L^2}^\frac{1}{2}\|{\bf \partial_1} \partial_2^k b_1\|_{L^2}^\frac{1}{2} \|\partial_1 \partial_2^{3-k} u_1\|_{L^2}^\frac{1}{2}
\|{\bf \partial_2} \partial_1 \partial_2^{3-k} u_1\|_{L^2}^\frac{1}{2}\|\partial_2^3 b_1\|_{L^2}^\frac{1}{2}\|{\bf \partial_3} \partial_2^3 b_1\|_{L^2}^\frac{1}{2}\\
& + |I_{4,2,2,3}|\\
\lesssim & \; \|b\|_{H^3} \|\na_h u\|_{H^3} \|\partial_1 b\|_{H^2}^\frac{1}{2} \|\partial_3 b\|_{H^3}^\frac{1}{2} +  |I_{4,2,2,3}|.
\end{split}
\end{equation}
By H\"{o}lder's inequality and $\p_2 b_2 = -\p_1 b_1 - \p_3 b_3$,
\begin{equation}\nonumber
\begin{split}
I_{5,2,1,2}
= & \sum_{k = 1}^2 \mathcal{C}_{3}^k \int_{\mathbb{R}^3} \partial_2^k b_1 \partial_1 \partial_2^{3-k} u_2\partial_2^3 b_2\; dx
+ \int_{\mathbb{R}^3} \partial_2^3 b_1 \partial_1 u_2\partial_2^3 b_2\; dx\\
\lesssim & \sum_{k = 1}^2 \|\partial_2^k b_1\|_{L^3}\| \partial_1 \partial_2^{3-k} u_2\|_{L^6}\|\partial_2^3 b_2\|_{L^2}
 + \|\partial_2^3 b_1\|_{L^2} \|\partial_1 u_2\|_{L^\infty}\|\partial_2^3 b_2\|_{L^2}\\
 \lesssim & \; \|b\|_{H^3} \|\partial_1 u\|_{H^3} (\|\partial_1 b\|_{H^2} + \|\partial_3 b\|_{H^3}).
\end{split}
\end{equation}
The last term $I_{5,2,1,3}$ contains a part that can not be directly handled,
\begin{equation}\nonumber
\begin{split}
I_{5,2,1,3}
= & \sum_{k = 1}^2 \mathcal{C}_{3}^k \int_{\mathbb{R}^3} \partial_2^k b_1 \partial_1 \partial_2^{3-k} u_3\partial_2^3 b_3\; dx
+ \int_{\mathbb{R}^3} \partial_2^3 b_1 \partial_1 u_3\partial_2^3 b_3\; dx\\
\lesssim & \sum_{k = 1}^2 \|\partial_2^k b_1\|_{L^2}^\frac{1}{2}\|{\bf \partial_1} \partial_2^k b_1\|_{L^2}^\frac{1}{2} \|\partial_1 \partial_2^{3-k} u_3\|_{L^2}^\frac{1}{2}
\|{\bf \partial_2} \partial_1 \partial_2^{3-k} u_3\|_{L^2}^\frac{1}{2}\\
&\quad \times \|\partial_2^3 b_3\|_{L^2}^\frac{1}{2}\|{\bf \partial_3} \partial_2^3 b_3\|_{L^2}^\frac{1}{2} + K_1\\
\lesssim & \; \|b\|_{H^3} \|\partial_h u\|_{H^3} \|\partial_1 b\|_{H^2}^\frac{1}{2} \|\partial_3 b\|_{H^3}^\frac{1}{2} + K_1
\end{split}
\end{equation}
where
\begin{equation}\nonumber
K_1 = \int_{\mathbb{R}^3} \partial_2^3 b_1 \partial_1 u_3\partial_2^3 b_3\; dx.
\end{equation}
It does not appear to be possible to give a direct estimate on $K_1$. As in the estimate of $I_{4,2,2,3}$, we use the special structure of the equation for $b$ in (\ref{mhd}) and make the substitution
$$
\partial_1 u_3 = \p_t b_3 + u\cdot\na b_3 -\p_3^2 b_3 - b\cdot\na u_3.
$$
Then $K_1$ can be rewritten as
\begin{equation}\nonumber
\begin{split}
K_1 =& \int_{\mathbb{R}^3} \Big( \partial_t b_3 + u\cdot \nabla b_3 - \partial_3^2 b_3 - b \cdot \nabla u_3 \Big) \partial_2^3b_1 \partial_2^3b_3 \; dx \\
=& \;  K_{1,1} +K_{1,2}+K_{1,3}+K_{1,4}.
\end{split}
\end{equation}
We estimate $K_{1,2}$, $K_{1,3}$ and $K_{1,4}$ similarly as $J_2, J_3, J_4$ to obtain
\begin{equation}\nonumber
\begin{split}
|K_{1,2}| \lesssim & \; \|u\|_{H^3}^\frac{1}{2}\|b\|_{H^3}^\frac{3}{2} \|\na_h u\|_{H^3}^\frac{1}{2} \|\partial_1 b\|_{H^2}^\frac{1}{2}\|\partial_3 b\|_{H^3},\\
|K_{1,3}| \lesssim & \; \|b\|_{H^3} \|\partial_3 b\|_{H^3}^2,\\
|K_{1,4}| \lesssim & \; \|u\|_{H^3}^\frac{1}{2} \|b\|_{H^3}^\frac{3}{2} \|\na_h u\|_{H^3}^\frac{1}{2}
\|\partial_1 b\|_{H^2}^\frac{1}{2} \|\partial_3 b\|_{H^3}.
\end{split}
\end{equation}
By integration by parts,
\begin{equation}\nonumber
\begin{split}
K_{1,1} = \frac{d}{dt}\int_{\mathbb{R}^3} b_3 \partial_2^3b_1 \partial_2^3b_3 \; dx - \int_{\mathbb{R}^3} b_3 \p_t (\partial_2^3b_1 \partial_2^3b_3 )\; dx =K_{1,1,1} + K_{1,1,2}.
\end{split}
\end{equation}
According to the equation of $b$ in (\ref{mhd}),
\begin{equation}\nonumber
\begin{cases}
\partial_2^3 \p_t b_{1} + \partial_2^3(u\cdot \nabla b_1) - \partial_2^3 \partial_3^2 b_1 =  \partial_2^3 (b \cdot \nabla u_1) + \partial_2^3 \partial_1 u_1,\\
\partial_2^3 \p_t b_{3} + \partial_2^3(u\cdot \nabla b_3) - \partial_2^3 \partial_3^2 b_3 =  \partial_2^3 (b \cdot \nabla u_3) + \partial_2^3 \partial_1 u_3.
\end{cases}
\end{equation}
Hence,
\begin{equation}\nonumber
\begin{split}
K_{1,1,2} = & \int_{\mathbb{R}^3} b_3 \partial_2^3[u\cdot \nabla b_1 -  \partial_3^2 b_1  - b \cdot \nabla u_1 - \partial_1 u_1] \partial_2^3 b_3\\
& + b_3 \partial_2^3[u\cdot \nabla b_3 -  \partial_3^2 b_3  - b \cdot \nabla u_3 - \partial_1 u_3] \partial_2^3 b_1\; dx\\
= & \int_{\mathbb{R}^3} b_3 [\partial_2^3 (u\cdot \nabla b_1) \partial_2^3 b_3 + \partial_2^3 (u\cdot \nabla b_3) \partial_2^3 b_1] - b_3 [\partial_2^3 \partial_3^2 b_1 \partial_2^3 b_3 + \partial_2^3 \partial_3^2 b_3 \partial_2^3 b_1]\\
& -  b_3 [\partial_2^3(b \cdot \nabla u_1) \partial_2^3 b_3 + \partial_2^3(b \cdot \nabla u_3) \partial_2^3 b_1]
- b_3[\partial_2 ^3 \partial_1 u_1 \partial_2^3 b_3 + \partial_2 ^3 \partial_1 u_3 \partial_2^3 b_1] \\
=& \; K_{1,1,2,1} + K_{1,1,2,2}+ K_{1,1,2,3}+ K_{1,1,2,4}.
\end{split}
\end{equation}
As in the estimate of the term $J_{1,2,1}$, we have
\begin{equation}\nonumber
K_{1,1,2,1} \lesssim \; \|u\|_{H^3}^\frac{1}{2} \|b\|_{H^3}^\frac{3}{2} \|\na_h u\|_{H^3}^\frac{1}{2} \|\partial_1 b\|_{H^2}^\frac{1}{2}\|\partial_3 b\|_{H^3} +
\|b\|_{H^3}^2 \|\na_h u\|_{H^3} \|\partial_1 b\|_{H^2}^\frac{1}{2} \|\partial_3 b\|_{H^3}^\frac{1}{2}.
\end{equation}
Similar to the terms $J_{1,2,2}$, $J_{1,2,3}$ and $J_{1,2,4}$, we have
\begin{equation}\nonumber
\begin{split}
K_{1,1,2,2} = &\int_{\mathbb{R}^3} 2 b_3 \partial_2^3 \partial_3 b_1 \partial_2^3 \partial_3 b_3 + \partial_3 b_3 [\partial_2^3 \partial_3 b_3 \partial_2^3 b_1 + \partial_2^3 \partial_3 b_1 \partial_2^3 b_3] \;dx\\
\lesssim & \; \|b\|_{L^\infty} \|\partial_3 b\|_{H^3}^2 + \|\partial_3 b_3\|_{L^6} \|\partial_3 b\|_{H^3} \|b\|_{H^3} \\
\lesssim &  \;\|b\|_{H^3} \|\partial_3 b\|_{H^3}^2,
\end{split}
\end{equation}
\begin{equation}\nonumber
\begin{split}
K_{1,1,2,3} \lesssim & \; \|b\|_{H^3}^2 \|\partial_2 u\|_{H^3} \|\partial_1 b\|_{H^2}^\frac{1}{2} \|\partial_3 b\|_{H^3}^\frac{1}{2}\\
&\; + \|u\|_{H^3}^\frac{1}{2} \|b\|_{H^3}^\frac{3}{2}
\|\partial_h u\|_{H^3}^\frac{1}{2} \|\partial_1 b\|_{H^2}^\frac{1}{2} \|\partial_3 b\|_{H^3}\\
&\; + \|b\|_{H^3}^2 \|\partial_2 u\|_{H^3} \|\partial_1 b\|_{H^2}^\frac{1}{2} \|\partial_3 b\|_{H^3}^\frac{1}{2}
\end{split}
\end{equation}
and
\begin{equation}\nonumber
K_{1,1,2, 4} \lesssim \; \|b\|_{H^3} \|\partial_1 u\|_{H^3} \|\partial_1 b\|_{H^2}^\frac{1}{2} \|\partial_3 b\|_{H^3}^\frac{1}{2}.
\end{equation}
Integrating (\ref{E0:eq2}) in time, namely
$$
E_0(t) \lesssim \; E_0(0) + \int_0^t (I_2(\tau) + I_3(\tau) + I_4(\tau) + I_5(\tau))\,d\tau
$$
and inserting all the bounds obtained above for $I_2$ through $I_5$, we obtain \eqref{E0} after applying H\"{o}lder's inequality. To be clear, we provide some details. The bounds for $I_2$ in (\ref{jj}) and (\ref{jj1}) yield
\beno
\int_0^t I_2(\tau)\,d\tau  &\lesssim&  \int_0^t \|u\|_{H^3} \, \|\na_h u\|_{H^3}^2\,d\tau \\
&\le& \sup_{0\le \tau\le t} \|u(\tau)\|_{H^3}\, \int_0^t \|\na_h u\|_{H^3}^2\,d\tau \le E_0^{\frac32}(t).
\eeno
The bounds for $I_3$ in (\ref{jj2}), (\ref{jj3}) and (\ref{jj4}) lead to, by H\"{o}lder's inequality,
\beno
\int_0^t I_3(\tau)\,d\tau  &\lesssim&  \int_0^t \|b\|_{H^3} \|\partial_1 b\|_{H^2} \|\partial_1 u \|_{H^3}\,d\tau\\
&& + \int_0^t \|b\|_{H^3} \|\partial_1 b\|_{H^2}^\frac{1}{2} \|\partial_3 b\|_{H^3}^\frac{1}{2} \|\partial_2 u \|_{H^3}\,d\tau \\
&&  +  \int_0^t \|b\|_{H^3}^{\frac12}\, \|u\|_{H^3}^{\frac12}\,\|\partial_1 b\|_{H^2}^\frac{1}{2}\, \|\partial_3 b\|_{H^3} \|\partial_2 u \|_{H^3}^{\frac12}\,d\tau\\
&\lesssim& E_0(t)^{\frac12} \,E_1(t)^{\frac12}\, E_0(t)^{\frac12} + E_0(t)^{\frac12}
E_1(t)^{\frac14}\, E_0(t)^{\frac14}\, E_0(t)^{\frac12}\\
&\lesssim& E_0^{\frac32}(t) + E_1^{\frac32}(t).
\eeno
The bounds for $I_4$ involves a lot of terms and we shall just choose some typical ones to bound the time integral of $I_4$. For example,  the time integrals of the bounds for $I_{4, 3}$ in (\ref{pp1}), $J_4$ in (\ref{pp2}) and $J_{1,1}$ in (\ref{pp3})
obeys, by H\"{o}lder's inequality,
\beno
\int_0^t  \|u\|_{H^3}^\frac{1}{2}\|b\|_{H^3}^\frac{1}{2}\|\partial_2 u\|_{H^3}^\frac{1}{2} \|\partial_3 b\|_{H^3}^\frac{3}{2}\,d\tau &\le& E_0(t)^{\frac12}\, E_0(t)^{\frac14}\, E_0(t)^{\frac34}\, =E_0^{\frac32}(t),
\eeno
\beno
\int_0^t \|u\|_{H^3}^\frac{1}{2}\|b\|_{H^3}^\frac{3}{2} \|\na_h u\|_{H^3}^\frac{1}{2} \|\partial_1 b\|_{H^2}^\frac{1}{2} \|\partial_3 b\|_{H^3}\,d\tau  &\le& E_0(t)\, E_0(t)^{\frac14}\, E_1(t)^{\frac14}\, E_0(t)^{\frac12}\\
&\lesssim& E_0^2(t) + E_1^2(t)
\eeno
and
\beno
\int_0^t J_{1,1} \,d\tau &=& 3 \int_{\mathbb{R}^3} b_1 (\partial_2 ^3 b)^2 \;dx - 3 \int_{\mathbb{R}^3} b_1(x,0) (\partial_2 ^3 b)^2(x,0) \;dx \\
&\lesssim& \|b_1(0)\|_{L^\infty}\, \|b(0)\|_{H^3}^2 +
\|b_1(t)\|_{L^\infty}\, \|b(t)\|_{H^3}^2\\
&\lesssim& E_0(0)^{\frac32} +  E_0^{\frac32}(t).
\eeno
The time integral of $I_5$ is similarly bounded. This completes the proof of (\ref{E0}).
\end{proof}

\vskip .3in
\section{Proof of (\ref{E1})}
\label{sec3}
\setcounter {equation}{0}

This section proves (\ref{E1}), namely
$$
E_1(t) \lesssim E_1(0) + E_0(t) + E_0(t)^\frac{3}{2} + E_1(t)^\frac{3}{2}.
$$

\begin{proof}[Proof of (\ref{E1})] Due to the equivalence of the norm
	$\|\p_1 b\|_{H^2}$ and the norm $\|\p_1 b\|_{L^2} + \|\p_1 b\|_{\dot{H}^2}$, it suffices to estimate the $L^2$-norm  and the homogeneous $\dot H^2$-norm of $\p_1b$. We make use of the velocity equation  in  \eqref{mhd} to write
\begin{equation}\nonumber
  \partial_1b=\p_tu + u\cdot\nabla u-\Delta_h u +\nabla P-b\cdot\nabla b .
\end{equation}
Therefore,
\ben
  \| \partial_1 b\|_{L^2}^2 &= &\int_{\mathbb{R}^3} \p_t u \cdot\partial_1 b \; dx + \int_{\mathbb{R}^3} u\cdot\nabla u \cdot\partial_1 b \; dx \notag\\
  &&-\int_{\mathbb{R}^3} \Delta_h u\cdot \partial_1 b \; dx -\int_{\mathbb{R}^3} b\cdot\nabla b \cdot\partial_1 b \; dx \notag\\
  &=& N_1+N_2+N_3+N_4, \label{l2es}
\een
where we have eliminated the pressure term due to $\na\cdot b=0$.
We integrate by parts and use the equation of $b$ in (\ref{mhd}) to obtain
\begin{equation}\nonumber
\begin{split}
  N_1 =& \frac{d}{dt} \int_{\mathbb{R}^3}u\cdot \partial_1 b\; dx - \int_{\mathbb{R}^3}u\cdot \partial_1(-u\cdot\nabla b +\partial_3^2 b+b\cdot\nabla u +\partial_1 u)dx \\
  = &N_{1,0}+N_{1,1}+N_{1,2}+N_{13}+N_{1,4}.
\end{split}
\end{equation}
By Lemma \ref{abound} and H\"{o}lder's inequality,
\begin{equation}\nonumber
\begin{split}
  N_{1,1} = &-\int_{\mathbb{R}^3} u\cdot(\partial_1u \cdot\nabla b +u\cdot\nabla \partial_1 b)dx \\
  \lesssim  & \|\partial_1u\|_{L^2}\|u\|_{L^2}^{\frac{1}{4}}\|\partial_1 u\|_{L^2}^{\frac{1}{4}}\|\partial_2 u\|_{L^2}^{\frac{1}{4}}\|\partial_1 \partial_2u\|_{L^2}^{\frac{1}{4}}\|\nabla b\|_{L^2}^{\frac{1}{2}}\|\nabla \partial_3b \|_{L^2}^{\frac{1}{2}} \\
  & + \|u\|_{L^2}^{\frac{1}{2}}\|\partial_1u\|_{L^2}^{\frac{1}{2}}\|u\|_{L^2}^{\frac{1}{2}}\|\partial_2u\|_{L^2}^{\frac{1}{2}}\|\nabla\partial_1b\|_{L^2}^{\frac{1}{2}}
  \|\nabla\partial_1\partial_3 b\|_{L^2}^{\frac{1}{2}} \\
  \leq & \|\na_h u\|_{H^3}^{\frac{3}{2}} \|\partial_3 b\|_{H^3}^{\frac{1}{2}} \|u\|_{H^3}^{\frac{1}{2}} \|b\|_{H^3}^{\frac{1}{2}} + \|\na_h u\|_{H^3} \|\partial_1 b\|_{H^2}\|u\|_{H^3},
\end{split}
\end{equation}
\begin{equation}\nonumber
  N_{1,2} = -\int_{\mathbb{R}^3} u \cdot \partial_1\partial_3^2 b\, dx \leq \|\na_h u\|_{H^3}\|\partial_3b\|_{H^3},
\end{equation}
\begin{equation}\nonumber
\begin{split}
  N_{1,3}= & \int_{\mathbb{R}^3} (u\cdot (\partial_1b \cdot\nabla u) + u \cdot (b\cdot\nabla \partial_1 u))\; dx \\
  \lesssim &   \|u\|_{L^2}^{\frac{1}{2}}\|\partial_1u\|_{L^2}^{\frac{1}{2}}\|\nabla u\|_{L^2}^{\frac{1}{2}}\|\partial_2\nabla u\|_{L^2}^{\frac{1}{2}}\|\partial_1b\|_{L^2}^{\frac{1}{2}}
  \|\partial_1\partial_3 b\|_{L^2}^{\frac{1}{2}} \\
  & +   \|u\|_{L^2}^{\frac{1}{2}}\|\partial_1u\|_{L^2}^{\frac{1}{2}}\|b\|_{L^2}^{\frac{1}{2}}\|\partial_3 b\|_{L^2}^{\frac{1}{2}}\|\nabla\partial_1u\|_{L^2}^{\frac{1}{2}}
  \|\nabla\partial_1\partial_2 u\|_{L^2}^{\frac{1}{2}} \\
  \leq & \|\na_h u\|_{H^3}\|\partial_1 b\|_{H^2}\|u\|_{H^3} + \|\na_h u\|_{H^3}^{\frac{3}{2}}\|\partial_3 b\|_{H^3}^{\frac{1}{2}} \|u\|_{H^3}^{\frac{1}{2}} \| b\|_{H^3}^{\frac{1}{2}}
\end{split}
\end{equation}
and
\begin{equation}\nonumber
  N_{1,4}= -\int_{\mathbb{R}^3} u \cdot \partial_1^2 u \; dx \leq  \|\na_h u\|_{H^3}^2.
\end{equation}
Similarly,
\begin{equation}\nonumber
\begin{split}
N_2 = &\int_{\mathbb{R}^3} u\cdot \nabla u \cdot\partial_1 b \; dx\\
\lesssim & \|u\|_{L^2}^\frac{1}{2} \|\partial_1 u\|_{L^2}^\frac{1}{2} \|\nabla u\|_{L^2}^\frac{1}{2}
\|\nabla \partial_2 u\|_{L^2}^\frac{1}{2} \|\partial_1 b\|_{L^2}^\frac{1}{2} \|\partial_1 \partial_3 b\|_{L^2}^\frac{1}{2} \\
\leq & \|\na_h u\|_{H^3}\|\partial_1b\|_{H^2}\|u\|_{H^3},
\end{split}
\end{equation}
\begin{equation}\nonumber
  N_3 = -\int_{\mathbb{R}^3} \Delta_h u \cdot\partial_1 b \; dx \leq \|\na_h u\|_{H^3} \|\partial_1 b\|_{H^2}
\end{equation}
and
\begin{equation}\nonumber
\begin{split}
  N_4 = &-\int_{\mathbb{R}^3} b\cdot\nabla b \cdot\partial_1 b \; dx \\
  \lesssim &\|b\|_{L^2}^\frac{1}{2} \|\partial_1 b\|_{L^2}^\frac{1}{2} \|\nabla b\|_{L^2}^\frac{1}{2}
\|\nabla \partial_3 b\|_{L^2}^\frac{1}{2} \|\partial_1 b\|_{L^2}^\frac{1}{2} \|\partial_1 \partial_2 b\|_{L^2}^\frac{1}{2} \\
\leq &\|\partial_1b\|_{H^2}^{\frac{3}{2}} \|\partial_3b\|_{H^3}^{\frac{1}{2}} \|b\|_{H^3}.
\end{split}
\end{equation}
This finishes the $L^2$ estimate of $\p_1 b$. Now we turn to the $\dot{H}^2$  estimate.  According to the equation of $u$ in (\ref{mhd}), we have
\ben
  \sum_{i=1}^3 \|\partial_i^2 \partial_1b \|_{L^2}^2 &= & \sum_{i=1}^3 \int_{\mathbb{R}^3} \partial_i^2 \p_t u \cdot \partial_i^2 \partial_1b\; dx \notag \\
  && + \sum_{i=1}^3 \int_{\mathbb{R}^3} \partial_i^2 (u\cdot\nabla u) \cdot  \partial_i^2 \partial_1b\; dx\notag \\
  &&-\sum_{i=1}^3 \int_{\mathbb{R}^3} \partial_i^2 \Delta_h u  \cdot  \partial_i^2 \partial_1b\; dx \notag\\
  && - \sum_{i=1}^3 \int_{\mathbb{R}^3} \partial_i^2 (b\cdot\nabla b) \cdot \partial_i^2 \partial_1b\; dx\notag \\
  &=& M_1+M_2+M_3+M_4. \label{h2es}
\een
To bound  $M_1$, we integrate by parts and use the equation of $b$ in (\ref{mhd}) to obtain
\begin{equation}\nonumber
\begin{split}
  M_1 = & \frac{d}{dt} \sum_{i=1}^3 \int_{\mathbb{R}^3} \partial_i^2 u \cdot \partial_i^2\partial_1b \; dx \\
  & - \sum_{i=1}^3 \int_{\mathbb{R}^3} \partial_i^2 u  \cdot \partial_i^2\partial_1 (-u\cdot\nabla b + \partial_3^2b+b\cdot\nabla u+\partial_1u)\; dx \\
  =& M_{1,0}+M_{1,1}+M_{1,2}+M_{1,3}+M_{1,4}.
\end{split}
\end{equation}
By Lemma \ref{abound},
\beno
  M_{1,1}&= & \int_{\mathbb{R}^3} \partial_1^2u  \cdot \partial_1^3(u\cdot\nabla b) + \partial_2^2u  \cdot \partial_2^2\partial_1(u\cdot\nabla b)+ \partial_3^3\partial_1u  \cdot \partial_3(u\cdot\nabla b) \; dx \\
  &\leq &  \| \na_h u\|_{H^3} \Big( \|\na_h u\|_{H^3}\|b\|_{H^3} + \|\partial_1b\|_{H^2}\|u\|_{H^3} \Big) \\
  &&+ \| \na_h u\|_{H^3} \Big( \|\partial_3u\|_{L^2}^{\frac{1}{4}} \|\partial_1\partial_3u\|_{L^2}^{\frac{1}{4}}\|\partial_2\partial_3u\|_{L^2}^{\frac{1}{4}}\|\partial_1\partial_2\partial_3u\|_{L^2}^{\frac{1}{4}}
  \|\nabla b\|_{L^2}^{\frac{1}{2}}\|\nabla\partial_3b\|_{L^2}^{\frac{1}{2}} \\
  &&\qquad + \|u\|_{L^2}^{\frac{1}{4}} \|\partial_1u\|_{L^2}^{\frac{1}{4}}\|\partial_2u\|_{L^2}^{\frac{1}{4}}\|\partial_1\partial_2u\|_{L^2}^{\frac{1}{4}}
  \|\partial_3\nabla b\|_{L^2}^{\frac{1}{2}}\|\nabla\partial_3^2b\|_{L^2}^{\frac{1}{2}} \Big) \\
  &\leq & \| \na_h u\|_{H^3} \Big( \|\na_h u\|_{H^3}\|b\|_{H^3} + \|\partial_1b\|_{H^2}\|u\|_{H^3} + \|\na_h u\|_{H^3}^{\frac{1}{2}} \|\partial_3b\|_{H^3}^{\frac{1}{2}}\|u\|_{H^3}^{\frac{1}{2}}\|b\|_{H^3}^{\frac{1}{2}}      \Big).
\eeno
Clearly,
$$
  M_{1,2} = -\sum_{i=1}^3 \int_{\mathbb{R}^3} \partial_1\partial_i^3 u\cdot \partial_i \partial_3^2 b \; dx \\
  \leq \|\na_h u\|_{H^3} \|\partial_3 b\|_{H^3} .
$$
By Lemma \ref{abound},
\begin{equation}\nonumber
\begin{split}
  M_{1,3} = & - \int_{\mathbb{R}^3} \partial_1^2u \cdot \partial_1^3(b\cdot\nabla u) + \partial_2^2u\cdot  \partial_2^2\partial_1(b\cdot\nabla u)+ \partial_3^3\partial_1u\cdot  \partial_3(b\cdot\nabla u) \; dx \\
  \leq & \| \na_h u\|_{H^3} \Big( \|\na_h u\|_{H^3}\|b\|_{H^3} + \|\partial_1b\|_{H^2}\|u\|_{H^3} \Big) \\
  &+ \| \na_h u\|_{H^3} \Big( \|\nabla u\|_{L^2}^{\frac{1}{4}} \|\partial_1\nabla u\|_{L^2}^{\frac{1}{4}}\|\partial_2\nabla u\|_{L^2}^{\frac{1}{4}}\|\partial_1\partial_2\nabla u\|_{L^2}^{\frac{1}{4}}
  \|\partial_3 b\|_{L^2}^{\frac{1}{2}}\|\partial_3^2b\|_{L^2}^{\frac{1}{2}} \\
  &\qquad + \|\nabla\partial_3u\|_{L^2}^{\frac{1}{4}} \|\partial_1\partial_3\nabla u\|_{L^2}^{\frac{1}{4}}\|\partial_2\partial_3\nabla u\|_{L^2}^{\frac{1}{4}}\|\partial_1\partial_2\partial_3\nabla u\|_{L^2}^{\frac{1}{4}}
  \|b\|_{L^2}^{\frac{1}{2}}\|\partial_3b\|_{L^2}^{\frac{1}{2}} \Big) \\
  \leq & \| \na_h u\|_{H^3} \Big( \|\na_h u\|_{H^3}\|b\|_{H^3} + \|\partial_1b\|_{H^2}\|u\|_{H^3} + \|\na_h u\|_{H^3}^{\frac{1}{2}} \|\partial_3b\|_{H^3}^{\frac{1}{2}}\|u\|_{H^3}^{\frac{1}{2}}\|b\|_{H^3}^{\frac{1}{2}}      \Big).
\end{split}
\end{equation}
Obviously,
$$
M_{1,4} \leq \sum_{i=1}^3 \int_{\mathbb{R}^3} \Big| \partial_i^2 \partial_1 u \Big|^2 \; dx \leq \|\na_h u\|_{H^3}^2.
$$
By Lemma \ref{abound},  $M_2$ is bounded by
\begin{equation}\nonumber
\begin{split}
M_2= & \int_{\mathbb{R}^3} \partial_1^2(u\cdot\nabla u)\cdot \partial_1^3 b + \partial_2^2(u\cdot\nabla u)\partial_2^2\cdot\partial_1 b +\partial_1(u\cdot\nabla u)\cdot\partial_3^4 b \; dx \\
\leq & \|\partial_1 b\|_{H^2}\|\na_h u\|_{H^3}\|u\|_{H^3} \\
&+ \| \partial_3 b\|_{H^3} \Big( \|\nabla u\|_{L^2}^{\frac{1}{4}} \|\partial_1\nabla u\|_{L^2}^{\frac{1}{4}}\|\partial_2\nabla u\|_{L^2}^{\frac{1}{4}}\|\partial_1\partial_2\nabla u\|_{L^2}^{\frac{1}{4}}
  \|\partial_1 u\|_{L^2}^{\frac{1}{2}}\|\partial_1\partial_3 u \|_{L^2}^{\frac{1}{2}} \\
  &\qquad + \|u\|_{L^2}^{\frac{1}{4}} \|\partial_1 u\|_{L^2}^{\frac{1}{4}}\|\partial_2 u\|_{L^2}^{\frac{1}{4}}\|\partial_1\partial_2 u\|_{L^2}^{\frac{1}{4}}
  \|\partial_1\nabla u\|_{L^2}^{\frac{1}{2}}\|\partial_1\partial_3\nabla u\|_{L^2}^{\frac{1}{2}} \Big) \\
  \leq & \|\partial_1 b\|_{H^2}\|\na_h u\|_{H^3}\|u\|_{H^3} + \|\partial_3 b\|_{H^3}\|\na_h u\|_{H^3}\|u\|_{H^3}.
\end{split}
\end{equation}
The bound for  $M_3$ is straightforward,
\begin{equation}\nonumber
  M_3 = \sum_{i=1}^3 \int_{\mathbb{R}^3} \partial_i^2 \Delta_h u \partial_i^2 \partial_1 b \; dx \leq \|\partial_h u\|_{H^3}\|\partial_1b\|_{H^2}.
\end{equation}
The last term  $M_4$ can be  bounded via Lemma \ref{abound},
\begin{equation}\nonumber
\begin{split}
  M_4 =&  \int_{\mathbb{R}^3} \partial_1^2(b\cdot\nabla b)\cdot \partial_1^3b+ \partial_3^2 (b\cdot\nabla b)\cdot \partial_3^2\partial_1b+\partial_2^2(b\cdot\nabla b)\cdot \partial_2^2\partial_1 b \; dx \\
  \leq & \|\partial_1b\|_{H^2}^2\|b\|_{H^3}+\|\partial_1b\|_{H^2}\|\partial_3b\|_{H^3} \|b\|_{H^3} \\
   &+ \| \partial_2^2\partial_1 b\|_{L^2} \Big( \|\nabla b\|_{L^2}^{\frac{1}{4}} \|\partial_1\nabla b\|_{L^2}^{\frac{1}{4}}\|\partial_2\nabla b\|_{L^2}^{\frac{1}{4}}\|\partial_1\partial_2\nabla b\|_{L^2}^{\frac{1}{4}}
  \|\partial_2^2 b\|_{L^2}^{\frac{1}{2}}\|\partial_2^2\partial_3 b \|_{L^2}^{\frac{1}{2}} \\
  &\qquad + \|\partial_2b\|_{L^2}^{\frac{1}{4}} \|\partial_1 \partial_2b\|_{L^2}^{\frac{1}{4}}\|\partial_2^2 b\|_{L^2}^{\frac{1}{4}}\|\partial_1\partial_2^2 b\|_{L^2}^{\frac{1}{4}}
  \|\partial_2\nabla b\|_{L^2}^{\frac{1}{2}}\|\partial_2\partial_3\nabla b\|_{L^2}^{\frac{1}{2}}  \\
  &\qquad + \|b\|_{L^2}^{\frac{1}{4}} \|\partial_1 b\|_{L^2}^{\frac{1}{4}}\|\partial_2 b\|_{L^2}^{\frac{1}{4}}\|\partial_1\partial_2 b\|_{L^2}^{\frac{1}{4}}
  \|\partial_2^2\nabla b\|_{L^2}^{\frac{1}{2}}\|\partial_2^2\partial_3\nabla b\|_{L^2}^{\frac{1}{2}} \Big) \\
\leq & \|\partial_1b\|_{H^2}^2\|b\|_{H^3}+\|\partial_1b\|_{H^2}\|\partial_3b\|_{H^3} \|b\|_{H^3} +\|\partial_1b\|_{H^2}^{\frac{3}{2}}\|\partial_3b\|_{H^3}^{\frac{1}{2}}\|b\|_{H^3}.
\end{split}
\end{equation}
Adding (\ref{l2es}) and (\ref{h2es}), integrating in time, invoking the bound for $N_1$ through $N_4$ and $M_1$ through $M_4$, and applying H\"{o}lder's inequality to the time integrals, we obtain (\ref{E1}). For the sake of clarity, we provide the details. The time integrals of $N_{1, 0}$ and $M_{1,0}$ are bounded by
\beno
\int_0^t N_{1,0}\,d\tau &=&
\int_{\mathbb{R}^3}u(x, t)\cdot \partial_1 b(x, t)\; dx
- \int_{\mathbb{R}^3}u(x, 0)\cdot \partial_1 b(x, 0)\; dx\\
 &\le& E_0(t) + E_0(0),
\eeno
\beno
\int_0^t M_{1,0}\,d\tau &=&  \sum_{i=1}^3 \int_{\mathbb{R}^3} \partial_i^2 u(x,t) \cdot \partial_i^2\partial_1b(x,t) \; dx \\
&&- \sum_{i=1}^3 \int_{\mathbb{R}^3} \partial_i^2 u(x,0)\cdot \partial_i^2\partial_1b(x,0) \; dx
\le E_0(t) + E_0(0).
\eeno
By H\"{o}lder's inequality,
$$
\int_0^t N_{1,1} d\tau \lesssim E_0^{\frac32}(t) + E_1^{\frac32}(t), \qquad \int_0^t N_{1,3}\,d\tau \lesssim E_0^{\frac32}(t) + E_1^{\frac32}(t).
$$
Clearly,
$$
\int_0^t N_{1,2}\,d\tau \le E_0(t), \qquad \int_0^t N_{1,4}\,d\tau \le E_0(t)
$$
and
$$
\int_0^t N_2\,d\tau \le E_0^{\frac32}(t) + E_1^{\frac32}(t),  \qquad
\int_0^t N_4\,d\tau \le E_0^{\frac32}(t) + E_1^{\frac32}(t).
$$
The integral of $N_3$ is slightly different. By H\"{o}lder's inequality,
$$
\int_0^t N_3\,d\tau \le E_0(t)^{\frac12} \, E_1(t)^{\frac12} \le \frac14 E_1(t) + C\,E_0(t).
$$
Furthermore,
$$
\int_0^t M_{1,1}\,d\tau,\,\,\int_0^t M_{1,2}\,d\tau  \lesssim E_0^{\frac32}(t) + E_1^{\frac32}(t), \quad \int_0^t M_{1,2}\,d\tau, \,\, \int_0^t M_{1,4}\,d\tau \lesssim E_0(t),
$$
$$
\int_0^t M_2\,d\tau \lesssim E_0^{\frac32}(t), \quad \int_0^t M_{4}\,d\tau\lesssim E_0^{\frac32}(t) + E_1^{\frac32}(t)
$$
and
$$
\int_0^t M_3\,d\tau \le E_0(t)^{\frac12} \, E_1(t)^{\frac12} \le \frac14 E_1(t) + C\,E_0(t).
$$
Combining all the bounds above yields
$$
E_1(t) \le E_0(0) + \frac12 E_1(t) + C\, E_0(t) + C\, E_0(t)^{\frac32} + C\, E_1(t)^{\frac32},
$$
which gives (\ref{E1}). This completes the proof of (\ref{E1}).
\end{proof}

\vskip .3in
\centerline{\bf Acknowledgments}

\vskip .1in
Jiahong Wu is partially supported by NSF grant DMS 1624146 and the AT\&T Foundation at Oklahoma State University.
Yi Zhu is partially supported by Shanghai Sailing Program (No.18YF1405500) and NSFC (No.11801175).

\end{document}